\documentclass[10pt]{article}

\usepackage{enumitem}
\usepackage{setspace}
\usepackage{verbatim}
\usepackage{amsthm,latexsym}

\newtheorem{theorem}{Theorem}[section]
\newtheorem{prop}[theorem]{Proposition}
\newtheorem{cor}[theorem]{Corollary}
\newtheorem{quest}[theorem]{Question}

\newtheorem{defn}[theorem]{Definition}
\newtheorem{lemma}[theorem]{Lemma}

\usepackage{tikz}
\usetikzlibrary{calc}
\tikzset{
  bigblue/.style={circle, draw=blue!80,fill=blue!40,thick, inner sep=1.5pt, minimum size=5mm},
  bigred/.style={circle, draw=red!80,fill=red!40,thick, inner sep=1.5pt, minimum size=5mm},
  bigblack/.style={circle, draw=black!100,fill=black!40,thick, inner sep=1.5pt, minimum size=5mm},
  bluevertex/.style={circle, draw=blue!100,fill=blue!100,thick, inner sep=0pt, minimum size=2mm},
  redvertex/.style={circle, draw=red!100,fill=red!100,thick, inner sep=0pt, minimum size=2mm},
  blackvertex/.style={circle, draw=black!100,fill=black!100,thick, inner sep=0pt, minimum size=2mm},  
  whitevertex/.style={circle, draw=black!100,fill=white!100,thick, inner sep=0pt, minimum size=2mm},  
  smallblack/.style={circle, draw=black!100,fill=black!100,thick, inner sep=0pt, minimum size=1mm},  
}

\newcommand{\ec}{$2$-edge-coloured }

\makeatletter
\renewcommand{\p@enumii}{\theenumi.}
\makeatother

\begin{document}

\title{Edge-coloured graph homomorphisms, paths, and duality}

\author{Kyle Booker and Richard C. Brewster\thanks{Both authors supported by NSERC}
\thanks{e-mail: rbrewster@tru.ca} \\
Dept. of Math and Stats\\ Thompson Rivers University\\ Kamloops, Canada}

%\address{Dept. of Math and Stats\\ Thompson Rivers University\\ Kamloops, CANADA}
%\email{rbrewster@tru.ca}

\date{}

\maketitle

\begin{abstract}
We present a \ec analogue of the duality theorem for transitive
tournaments and directed paths.  Given a \ec path $P$ whose edges alternate blue
and red, we construct a \ec graph $D$ so that for any \ec graph $G$
$$
P \to G \Leftrightarrow G \not\to D.
$$
The duals are simple to construct, in particular  $|V(D)|=|V(P)|-1$.
\end{abstract}

\bigskip
\centerline{\emph{Dedicated to Gary MacGillivray on the occasion of his $60^{th}$ birthday.}}

\section{Introduction}

In this paper we study homomorphisms of edge-coloured graphs. Our main
result is a \ec analogue of the duality theorem for transitive
tournaments and directed paths.  While duality theorems for finite structures
are fully described by Ne\v{s}et\v{r}il and Tardif~\cite{DualTheorems}, 
our contribution is to present a simple construction of small (linear) duals
for \ec alternating paths,
and use them to highlight similarities and differences between edge-coloured 
graph and digraph homomorphisms.

A \emph{\ec graph $G$} is a vertex set $V(G)$ together with 
$2$ edge sets $E_1(G), E_2(G)$.  We allow loops 
and multiple edges provided parallel edges have different colours. 
The \emph{underlying graph of $G$} is the (classical)
graph with vertices $V(G)$ and edge set $E_1(G) \cup E_2(G)$, 
i.e. the graph obtained by ignoring edge colours and deleting multiple edges. 
A \ec graph is a \emph{path} if the underlying 
graph is a path. We similarly define other standard notions 
like cycle and walk.

As a convention colours 1 and 2 are blue and red respectively.
Blue edges are depicted in diagrams as solid while red edges are dashed. 
A vertex is \emph{mixed} if it is incident with both blue and red edges.
A vertex is \emph{blue only} (respectively \emph{red only}) if it is incident
only with blue (respectively red) edges.  See for example Figure~\ref{fig:Ds}:
the vertices labelled $+1$ are blue only and vertices labelled $+2$ are mixed.

Let $G$ and $H$ be \ec graphs.  
A \emph{homomorphism} $\varphi$ of $G$ to $H$ is a mapping $\varphi: V(G) \to V(H)$
such that $uv \in E_i(G)$ implies $\varphi(u) \varphi(v) \in E_i(H)$. We write
$\varphi: G \to H$ to indicate the existence of a homomorphism or simply $G \to H$
when the name of the mapping is not important.  
Homomorphisms of digraphs and relational systems are similarly defined.
If $G \to H$ and $H \to G$ we say $G$ is \emph{homomorphically equivalent} to $H$ 
and write $G \sim H$.

Finally, we remark that the above definitions naturally generalize to $k$-edge coloured
graphs.

The launching point for our work is the well known digraph homomorphism duality
theorem for transitive tournaments. Let $T_n$ be the \emph{transitive tournament on $n$ 
vertices} with vertex set $\{ 1, 2, \dots, n \}$ and arc set 
$\{ ij | 1 \leq i < j \leq n \}$. Let $\vec{P}_n$ be the directed path on $n$
vertices. Given a digraph $G$, then
$$
G \not\to T_n \Leftrightarrow \vec{P}_{n+1} \to G
$$
The digraphs $(\vec{P}_{n+1}, T_n)$ are a \emph{duality pair}. 
In general, given a relational system $F$, the pair $(F, D(F))$ is a \emph{duality pair}
if for all relational systems $G$, $G \not\to D(F)$ if and only if $F \to G$.
More generally, $\mathcal{F} = \{F_1, F_2, \dots, F_t\}$ and $H$ have 
\emph{finite duality} if for any $G$
$$
G \not\to H \Leftrightarrow F_i \to G \mbox{ for some } i.
$$
The pair $( \mathcal{F}, H)$ is a \emph{finite homomorphism duality}. 
In~\cite{bearconstruction}, Ne\v{s}et\v{r}il and Tardif
prove the existence of finite homomorphism dualities if and only if 
$\mathcal{F}$ is a set whose underlying graphs are trees.  
Their construction produces exponentially large duals.  In some cases
these duals are homomorphically equivalent to a much small structure, 
but there are trees $F$ with exponentially large dual $D(F)$
that cannot be reduced to a smaller structure~\cite{DualTheorems}. 

The motivation for this work is to explicitly (and easily) construct the duals 
for \ec alternating paths (defined below).  This is motivated by the fact that there are many 
parallels to the theory of \ec graph homomorphisms and digraph homomorphisms.
Specifically, homomorphisms of \ec graphs to alternating paths
behave like homomorphisms of digraphs to directed paths.  This similarity has been 
observed in multiple works~\cite{AM98,Brewster94,Brewster17,NesetrilRaspaud00}.
In this paper we now study homomorphisms of alternating paths into \ec graphs with the
hope that their dual \ec targets behave like transitive tournaments.
Our work elicits similarities and differences between the two families.

\section{Alternating path duality}

A \ec graph $G$ is an \emph{alternating path} if $G$ is a path whose edges 
successively alternate blue and red. 

\begin{defn}
The \ec graphs $F^B_k$ and $F^R_k$ are alternating paths 
of length $k$ with vertices $v_0, v_1, v_2, \dots, v_k$.  
When $k$ is odd the edge $v_{\lfloor{k/2} \rfloor}v_{\lceil{k/2} \rceil}$ 
is colour blue in $F^B_k$ and colour red in $F^R_k$. 
For even $k$, $F^B_k$ and $F^R_k$ are isomorphic alternating paths 
and we simply write $F_k$ to denote either.  
We define the families $\mathcal{F}_k := \{ F^R_k, F^B_k \}$ and 
$\mathcal{F} = \bigcup_{k=1}^\infty \mathcal{F}_k$.
\end{defn}

The first family of dual \ec graphs is defined as follows. 
See Figure~\ref{fig:Ds} for examples. 

\begin{defn}
Let $k \geq 1$ be an integer and let $j = \lfloor k/2 \rfloor$.  
The \ec graph $D_k$ has vertex set:
$$
V(D_k) = \left\{ \begin{array}{ll}
               \{ \pm 1, \pm 2, \dots, \pm j \} & \mbox{ if $k$ is even} \\
               \{ 0, \pm 1, \pm 2, \dots, \pm j \} & \mbox{ if $k$ is odd.}
               \end{array} \right.  
$$
There are blue loops on $\{ 1, 2, \dots, j \}$, red loops on $\{ -1, -2, \dots, -j \}$, 
and an edge $rs$ for all $|r| < |s|$. 
The edge $rs$ is blue if $r > 0$ and red if $r < 0$. 
If $r = 0$, the edge $rs$ is blue if $s > 0$ and red if $s < 0$. 
For odd values of $k$, the \ec graph $D^B_k$ (respectively $D^R_k$) is obtained 
from $D_k$ by adding a blue loop (respectively red loop) to the vertex $0$. 
We define the families 
$$
\mathcal{D}_k := \left\{ \begin{array}{ll}
  \{ D_k, D^B_k, D^R_k \} & \mbox{ if $k$ is odd,} \\
  \{ D_k \} & \mbox{ if $k$ is even,}
  \end{array} \right.
$$
and $\mathcal{D} = \bigcup_{k=1}^\infty \mathcal{D}_k$.
\end{defn}

\noindent
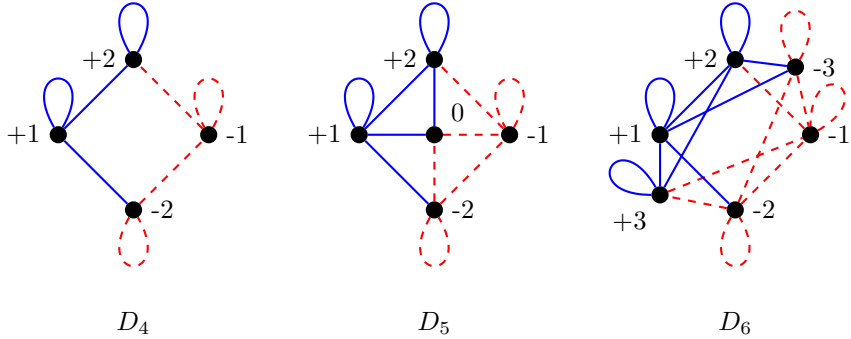
\begin{figure}%[!htpb]
\centering
  \begin{tikzpicture}[every loop/.style={},scale=1]
  \node[blackvertex,label={180:+1}] (b1) at (0,1) {};
  \path[thick,blue] (b1)   edge[out=60,in=120,loop, min distance=10mm] node  {} (b1);

  \node[blackvertex,label={180:+2}] (b2) at (1,2) {};
  \path[thick,blue] (b2)   edge[out=60,in=120,loop, min distance=10mm] node  {} (b2);

  \node[blackvertex, label={0:-1}] (r1) at (2,1) {};
  \path[thick,red,dashed] (r1)   edge[out=60,in=120,loop, min distance=10mm] node  {} (r1);
  
  \node[blackvertex, label={0:-2}] (r2) at (1,0) {}; 
  \path[thick,red,dashed] (r2)   edge[out=240,in=300,loop, min distance=10mm] node  {} (r2);
  
  \draw[thick,red,dashed] (r2)--(r1)--(b2);
  \draw[thick,blue] (b2)--(b1)--(r2);
  \node at (1,-1.5) {$D_4$};
  
  \begin{scope}[xshift=4cm]
  \node[blackvertex,label={180:+1}] (b1) at (0,1) {};
  \path[thick,blue] (b1)   edge[out=60,in=120,loop, min distance=10mm] node  {} (b1);

  \node[blackvertex,label={180:+2}] (b2) at (1,2) {};
  \path[thick,blue] (b2)   edge[out=60,in=120,loop, min distance=10mm] node  {} (b2);

  \node[blackvertex, label={0:-1}] (r1) at (2,1) {};
  \path[thick,red,dashed] (r1)   edge[out=60,in=120,loop, min distance=10mm] node  {} (r1);
  
  \node[blackvertex, label={0:-2}] (r2) at (1,0) {}; 
  \path[thick,red,dashed] (r2)   edge[out=240,in=300,loop, min distance=10mm] node  {} (r2);
  
  \node[blackvertex, label={30:0}] (z) at (1,1) {}; 
  
  \draw[thick,red,dashed] (r2)--(r1)--(b2) (r2)--(z)--(r1);
  \draw[thick,blue] (b2)--(b1)--(r2) (b2)--(z)--(b1);
  \node at (1,-1.5) {$D_5$};
  \end{scope}
  
  \begin{scope}[xshift=8cm]
  \node[blackvertex,label={180:+1}] (b1) at (0,1) {};
  \path[thick,blue] (b1)   edge[out=60,in=120,loop, min distance=10mm] node  {} (b1);

  \node[blackvertex,label={180:+2}] (b2) at (1,2) {};
  \path[thick,blue] (b2)   edge[out=60,in=120,loop, min distance=10mm] node  {} (b2);

  \node[blackvertex,label={240:+3}] (b3) at (0,0.2) {};
  \path[thick,blue] (b3)   edge[out=120,in=180,loop, min distance=10mm] node  {} (b3);

  \node[blackvertex, label={0:-1}] (r1) at (2,1) {};
  \path[thick,red,dashed] (r1)   edge[out=30,in=90,loop, min distance=10mm] node  {} (r1);
  
  \node[blackvertex, label={0:-2}] (r2) at (1,0) {}; 
  \path[thick,red,dashed] (r2)   edge[out=240,in=300,loop, min distance=10mm] node  {} (r2);
  
  \node[blackvertex, label={0:-3}] (r3) at (1.8,1.9) {};
  \path[thick,red,dashed] (r3)   edge[out=60,in=120,loop, min distance=10mm] node  {} (r3);
  
  \draw[thick,red,dashed] (r2)--(r1)--(b2) (r2)--(r3)--(r1) (r2)--(b3)--(r1);
  \draw[thick,blue] (b2)--(b1)--(r2) (b2)--(b3)--(b1) (b2)--(r3)--(b1);
  \node at (1,-1.5) {$D_6$};
  \end{scope}
\end{tikzpicture}
\caption{The targets $D_4, D_5, D_6$}\label{fig:Ds}
\end{figure}

Alternatively, one can describe the \ec graphs $D_k$ with a recursive construction.
The \ec graph $D_1$ is a single vertex: $V(D_1) = \{ 0 \}$.
The \ec graph $D_2$ has two vertices: $V(D_2) = \{ -1, 1 \}$.  There is a blue loop 
on $1$ and a red loop on $-1$.  For $k \geq 3$, the \ec graph $D_{k}$ is obtained from 
$D_{k-2}$ by adding two vertices $\{ -j, j \}$ where $j = \lfloor k/2 \rfloor$.  
There is a blue loop on $j$ and red loop on $-j$.  
For a vertex $v \in V(D_{k-2})$ there is a blue edge to each 
of $j$ and $-j$ if $v > 0$; a red edge to each if $v < 0$; and in the case $k$ is
odd there is a blue edge from $j$ to $0$ and a red edge from $-j$ to $0$.

The following proposition is immediate from the definitions.

\begin{prop}\label{prop:subgraphs}
For $k \geq 1$, $D_{2k-1} \subseteq D_{2k-1}^{c_0} \subseteq D_{2k} \subseteq D_{2k+1}$ 
and $F_k^{c_1} \subseteq F_{k+1}^{c_2}$
where $c_0, c_1, c_2 \in \{ B, R \}$.
\end{prop}

Since $G \subseteq H$ corresponds precisely to an embedding homomorphism $G \to H$, 
the proposition implies $D \to D'$ whenever $D \in \mathcal{D}_k$ and $D' \in \mathcal{D}_{k'}$ 
with $k < k'$. Also, $F_{k}^{c_1} \to F_{k'}^{c_2}$ for $k < k'$ and $c_1, c_2 \in \{ B, R \}$.
The homomorphism partial order on $\mathcal{D}$ under $\to$ is shown in Figure~\ref{fig:Poset},
which also includes the maps $D_k \to D^{c}_{k}$.

\begin{figure}%[!htpb]
\begin{center}
\begin{tikzpicture}[thick,scale=0.5, every node/.style={scale=1}]

\node (v1) at (0,0) {$D_1$};
\node (v1b) at (-2,2) {$D_1^B$};
\node (v1r) at (2,2) {$D_1^R$};
\node (v2) at (0,4) {$D_2$};
\node (v3) at (0,6) {$D_3$};
\node (v3b) at (-2,8) {$D_3^B$};
\node (v3r) at (2,8) {$D_3^R$};
\node (v4) at (0,10) {$D_4$};
\node (v5) at (0,12) {$D_5$};

\draw[thick,arrows=->] (v1)--(v1b);
\draw[thick,arrows=->] (v1)--(v1r); 
\draw[thick,arrows=->] (v1b)--(v2);
\draw[thick,arrows=->] (v1r)--(v2);
\draw[thick,arrows=->] (v2)--(v3);
\draw[thick,arrows=->] (v3)--(v3r); 
\draw[thick,arrows=->] (v3)--(v3b);
\draw[thick,arrows=->] (v3b)--(v4);
\draw[thick,arrows=->] (v3r)--(v4); 
\draw[thick,arrows=->] (v4)--(v5);

\node (dots) at (0,13) {$\vdots$};

\begin{scope}[xshift=10cm]
\node (f1) at (0,0) {$\{ F_1^B, F_1^R \}$};
\node (f1b) at (-2,2) {$F_1^R$};
\node (f1r) at (2,2) {$F_1^B$};
\node (f2) at (0,4) {$F_2$};
\node (f3) at (0,6) {$\{ F_3^B, F_3^R \}$};
\node (f3b) at (-2,8) {$F_3^R$};
\node (f3r) at (2,8) {$F_3^B$};
\node (f4) at (0,10) {$F_4$};
\node (v5) at (0,12) {$\{ F_5^B, F_5^R \}$};

\node (dots) at (0,13) {$\vdots$};

\end{scope}

\end{tikzpicture}
\end{center}
\caption{Homomorphism order for $\mathcal{D}_{k}$ and 
their duals.  Each dual admits a homomorphism to a dual 
with a larger subscript.}\label{fig:Poset}
\end{figure}
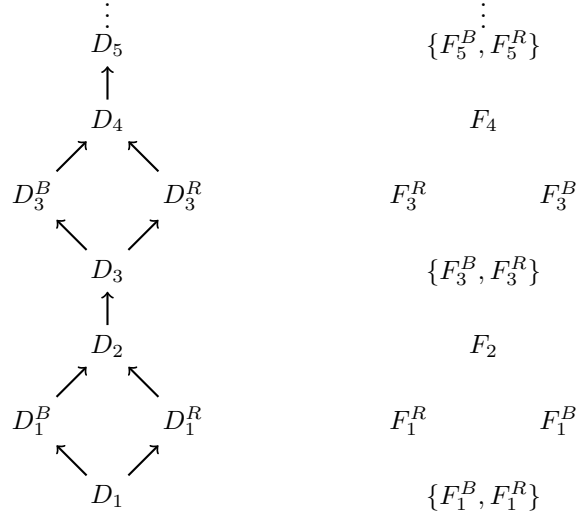

Our main results are the following duality theorems. The first covers 
duality pairs of the form $(F,D(F))$.
The second gives finite duality theorems for $(\mathcal{F}_k,D_k)$, 
specifically for odd $k$ there is a family of two obstructions.

\begin{theorem}\label{thm:main1}
Let $G$ be a \ec graph and $k \geq 1$ be an integer.  Then the following hold:
\begin{enumerate}[label=(\roman*)]
  \item $G \to D_{2k}$ if and only if $F_{2k} \not\to G$, 
  \item $G \to D^B_{2k-1}$ if and only if $F^R_{2k-1} \not\to G$, and
  \item $G \to D^R_{2k-1}$ if and only if $F^B_{2k-1} \not\to G$.
\end{enumerate}    
\end{theorem}

\begin{theorem}\label{thm:main2}
Let $G$ be an \ec graph and $k \geq 1$ be an integer.  
Then $G \to D_k$ if and only if for all $F \in \mathcal{F}_k, F \not\to G$.
\end{theorem}

If $G \to D$ for some $D \in \mathcal{D}$, then there is a (unique) minimal element
$D' \in \mathcal{D}$ such that $G \to D'$.  (This follows from our work below, but
Figure~\ref{fig:Poset} strongly suggests the result.) A consequence of the theorems above is
that we can certify the minimality using homomorphisms of the form $F \to G$ for 
$F \in \mathcal{F}$.  That is, we certify that $G \not\to D''$ for any predecessor of $D'$.
Formally,
\begin{defn}
Suppose $G \to D$ for some $D \in \mathcal{D}_k$.  A \emph{certificate of minimality} is:
\begin{itemize}
  \item a homomorphism $f: F_{k-1} \to G$ when $D = D_{k}$ and $k$ is odd;
  \item a homomorphism $f: F_{k}^R \to G$  when $D = D_{k}^R$ and $k$ is odd;
  \item a homomorphism $f: F_{k}^B \to G$  when $D = D_{k}^B$ and $k$ is odd; or
  \item a pair of homomorphisms $f_R: F_{k-1}^R \to G$ and $f_B: F_{k-1}^B \to G$
  when $D = D_{k}$ and $k$ is even. 
\end{itemize}
\end{defn}
In the following section we give a linear time algorithm that takes as input
a \ec graph $G$, determines the minimum $k$ for which $G \to D$ for some $D \in \mathcal{D}_k$
(if it exists), and produces a certificate of minimality.  (In an abuse of notation
to simplify the presentation of the algorithm, 
we write it returns $f: F \to G$ as the certificate of minimality.
In the last case listed above, the certificate is technically 2 maps, but could
be equivalently encoded as a single map $f: \{ F_{k-1}^R \cup F_{k-1}^B \} \to G$.)

The proof of both theorems is accomplished as follows.  The following lemma
proves one direction for each of the duality claims.  The converse will follow
from our algorithm.

\begin{lemma}\label{lem:FnotMap}
Let $k \geq 1$ and $F \in \mathcal{F}_k$.  Then $F \not\to D_k$.  Further if
$k$ odd, then $F^B_k \not\to D^R_k$ and $F^R_k \not\to D^B_k$.
\end{lemma}

\begin{proof}
We proceed by induction on $k$.  For $k = 1$ and $k=2$ the results are trivial.
Let $k > 2$ and assume to the contrary that there exists a homomorphism $f: F_k \to D_k$.  
Let $F'$ be the subpath of $F$ induced by $\{ v_1, \dots, v_{k-1} \}$.
Note that each vertex of $F'$ is mixed in $F$.  Thus under $f$ each vertex of $F'$
must map to a mixed vertex of $D_k$.  Let $D'$ be the subgraph of $D_k$ induced by the
mixed vertices, specifically $D' = D_k \backslash \{ \pm 1 \}$.
However, $F' \to D'$ contradicts the inductive hypothesis as 
$F' \in \mathcal{F}_{k-2}$ and $D'$ is isomorphic to $D_{k-2}$.  

A similar argument shows for $k$ odd, $F^B_k \not\to D^R_k$ and 
$F^R_k \not\to D^B_k$.
\end{proof}

A digraph containing a directed cycle does not admit a homomorphism to
$T_n$ for any $n$.  Analogously, a \ec graph $G$ containing a closed, alternating walk 
has the property that $F \to G$ for all $F \in \mathcal{F}$. Consequently $G$
does not admit a homomorphism to $D_k$ for any $k$.  Thus a closed alternating walk 
in $G$ certifies $G \not\to D_k$ for any $k$.  A \ec graph $G$ is 
\emph{smooth} if each vertex is mixed.

\begin{lemma}\label{lem:closedWalk}
Let $G$ be a smooth \ec graph.  Then $G$ contains a closed alternating walk
and $F_k \to G$ for all $k \geq 1$.  Consequently, $G \not\to D_k$ for any $k$.
\end{lemma}

\begin{proof}
Let $G$ be a smooth \ec graph.  Let $P$ be a maximal alternating path in $G$.  
Let the vertices of the path be $p_0, p_1, \dots, p_t$.  Suppose $p_0 p_1$
is blue.  By smoothness, $p_0$ is 
incident with a red edge.  By the maximality of $P$, if $p_0 u$ is red, then
$u = p_i$ for some $i$. If $i$ odd, then $p_0, p_1, \dots, p_i, p_0$ 
is an alternating cycle and we are done.  

Thus suppose $i$ is even. Similarly, $p_t$ has a neighbour
$p_j$ such that $p_{t-1} p_t p_j$ is an alternating path. 
In the case $j$ and $t$ have opposite parity, 
$p_j, p_{j+1}, \dots, p_{t-1}, p_t, p_j$ 
is an alternating cycle.  Otherwise we have
$p_0, p_i, p_{i+1}, \dots, p_t, p_{j}, p_{j-1}, \dots, p_1, p_0$ 
is a closed alternating walk.

The claim $F_k \to G$ for all $k \geq 1$ is trivial as one can simply wrap $F_k$ around
the closed walk.  By Lemma~\ref{lem:FnotMap}, $G \not\to D_k$ for any $k$.
\end{proof}

Clearly the proof of the lemma gives an algorithm for finding a closed alternating
walk in a smooth \ec graph.  (The path $P$ can be constructed greedily starting 
at an arbitrary vertex $v$.)
This is used in Step~\ref{step:Smooth} of our algorithm.

\section{Duality Algorithm}

We now present our algorithm that takes as input a \ec graph $G$, and
returns either a homomorphism $g: G \to D$, $D \in \mathcal{D}$ 
together with a certificate of minimality for $D$, or a smooth subgraph of $G$.  
In the latter case the smooth subgraph certifies $G \not\to D$ for all $D \in \mathcal{D}$.
(See Figure~\ref{fig:dualityalgorithm} for the algorithm pseudocode.)
  
Roughly, the algorithm works by mapping all blue (respectively red) only vertices of 
$G$ to +1 (respectively -1).  The mapped vertices are deleted from
$G$ and the process is repeated mapping the now blue (red) only vertices 
to +2 (-2). Delete the mapped vertices and iterate until the entirety of $G$ is 
mapped to $D_{2i}$ or a smooth subgraph is discovered.
 
The algorithm has a post-processing phase in Step~\ref{postProc}.
This step finds the minimum $D \in \{ D_{2i-1}, D_{2i-1}^B, D_{2i-1}^R, D_{2i} \}$
to which $G$ maps. A small note on the word minimum is required here.  At the termination
of the algorithm in Figure~\ref{fig:dualityalgorithm} 
we know that $G$ maps to some subset of the four targets above.
The only way this subset could not have a well defined minimum (see Figure~\ref{fig:Poset}) 
is if $G \to D_{2i-1}^B$ and $G \to D_{2i-1}^R$ but $G \not\to D_{2i-1}$.  However,
by Corollary~\ref{cor:homequivD}, this is impossible.

The final step of the algorithm is the construction of the certificate of minimality in
a subroutine Build($f$).  We believe the subroutine is more easily described in words
rather than pseudocode.  The paragraph below labelled Build($f$) contains its description.
Note we may assume that $G$ is connected as we can apply the algorithm 
on each component of $G$. 

\begin{figure}
\label{fig:dualityalgorithm}
  \hrulefill
  
  \centerline{\bf Duality Algorithm}
  
\begin{description}
  \item[Input:] A connected \ec graph $G$.
%  \item[Output:] Either $g: G \to D_{k}$ and $f: F \to G$ where $F \in \mathcal{F}_{k-1}$
%or a closed alternating walk $W$ in $G$.
  \item[Output:] Either $g: G \to D$ where $D \in \mathcal{D}$ together  
  with a certificate of minimality $f: F \to G$, 
%  $f: F \to G$ where $F \in \mathcal{F}_{2k-1} \cup \{F_{2k-2}\}$ for some positive integer $k$
or a closed alternating walk $W$ in $G$.
\end{description}

\begin{list}{\arabic{enumi}.}{\usecounter{enumi}}
\item \textbf{Set} $i=0$, $G_0 = G$.  \hspace{1cm} (\emph{$G_0$ is the unmapped subgraph})
\item \textbf{While} $(G_0 \neq\emptyset)$ \label{mainLoop}
   \begin{list}{\arabic{enumi}.\arabic{enumii}.}{\usecounter{enumii}}
     \item $i++$
     \item \textbf{Let} $B_i = \{u \in V(G_0)|u \mbox{ is blue only in } G_0 \}$.
     \item \textbf{Let} $R_i = \{u \in V(G_0)|u \mbox{ is red only in } G_0 \}$.
     \item \textbf{Let} $I_i = \{u \in V(G_0)|u \mbox{ is isolated in } G_0 \}$.
     \item\label{step:Smooth} \textbf{If} $B_i \cup R_i \cup I_i = \emptyset$, \textbf{then $G_0$}
     is smooth.  Find a closed alternating walk $W$ in $G_0$.  \textbf{Return} $W$ and 
     \textbf{answer NO}.
     \item \textbf{For each} $u \in B_i \cup I_i:$ \textbf{Set} $g(u) = i$. 
%     \textbf{If} $i == 1$, $Parent(u) = NULL$ \textbf{else} $Parent(u) = v$
%     where $v \in R_{i-1}$ and $uv \in E(G_0)$.
     \item \textbf{For each} $u \in R_i:$ \textbf{Set} $g(u) = -i$.
%     \textbf{If} $i == 1$, $Parent(u) = NULL$ \textbf{else} $Parent(u) = v$
%     where $v \in B_{i-1}$ and $uv \in E(G_0)$.
     \item $G_0=G_0-(B_i \cup R_i \cup I_i)$
   \end{list}
   \textbf{End while}
\item\label{postProc} \textbf{If} $B_i \cup R_i = \emptyset$, \textbf{set} $g(u) = 0$ 
for all $u \in I_i$, \textbf{return} $g: G \to D_{2i-1}$ \\
\textbf{Elseif} $R_i = \emptyset$, \textbf{set} $g(u) = 0$ 
for all $u \in B_i \cup I_i$, \textbf{return} $g: G \to D^B_{2i-1}$ \\
\textbf{Elseif} $B_i = \emptyset$, \textbf{set} $g(u) = 0$ 
for all $u \in R_i \cup I_i$, \textbf{return} $g: G \to D^R_{2i-1}$ \\
\textbf{ Else return} $g: G \to D_{2i}$.
\item\label{callBuild} \textbf{Call} Build(f)
\item \textbf{Return} $g, f$ and \textbf{answer YES}. 
\end{list}
\hrulefill
  
\caption{The algorithm}
\end{figure}

The correctness of the algorithm follows from some straightforward observations.
Using the notation as defined in Algorithm~1 (Figure~\ref{fig:dualityalgorithm}), 
let $G_{i}$ be the subgraph of $G$ induced by 
$\bigcup_{j=1}^i (B_{j} \cup R_{j} \cup I_{j})$.

\begin{prop}\label{prop:Invariant}
After $i$ iterations of the loop (Step~\ref{mainLoop}) the following invariants are 
true.
\begin{enumerate}
  \item The map $g$ defines a homomorphism $G_i \to D_{2i}$.
  \item Each vertex $v \in B_i$ (respectively $v \in R_i$) is
  the terminus of an alternating path of length $i-1$ whose last edge 
  is coloured red (respectively coloured blue), whereas each vertex $v\in I_i$
  is the midpoint of an alternating path of length $2i-2$.
\end{enumerate}
\end{prop}

\begin{proof}
Invariant~2 is trivial when $i=1$.  Assume $i > 1$.
Observe for $v \in B_i$, $v$ is incident with only blue edges during iteration $i$, 
but was mixed at any previous iteration, in particular at iteration $i-1$.  Hence $v$ is
adjacent to some $u \in R_{i-1}$.  
It follows by induction $v$ is the terminus of an alternating
path of length $i-1$ whose last edge is red.  The case $v \in R_i$ is analogous.  For
a vertex $v \in I_i$, observe that $v$ was mixed at iteration $i-1$, but isolated at
iteration $i$.  Hence, $v$ is adjacent to a vertex in $R_{i-1}$ and $B_{i-1}$.  
Thus Invariant~2 holds.

To see Invariant~1 holds, 
let $uv$ be an edge of $G_i$.  We may assume $u$ is mapped at iteration $i$.
Suppose $g(u)=i$ and $g(v) = j$. (The case $g(u)=-i$ is similar.) Since $u \in B_i \cup I_i$,
if $i = |j|$, then it must be the case that $u, v \in B_i$ and the edge $uv$ 
maps to the loop $ii$.  Otherwise, $|j| < i$.  If $j <0$, then $v \in R_j$; otherwise $v \in B_j$.
In the former case $uv$ is red (as $v$ is red only at iteration $j$).  In the
latter case $uv$ is blue.  The definition of $D_{2i}$ states that $ji$ is 
red for $j < 0$ and blue for $j > 0$.  Thus, $g$ is a homomorphism.
\end{proof}

Observe that the alternating paths described in Invariant 2 can be computed in linear time by
simply storing a pointer from each vertex to a parent whose deletion in the previous
iteration causes the vertex to no longer be mixed.

From Proposition~\ref{prop:Invariant}, 
at the terminus of Step~\ref{mainLoop} we have $g: G \to D_{2i}$.
The post processing at Step~\ref{postProc} is straightforward to analyze.  We complete the
proof of correctness for the algorithm with the subroutine Build(f).

\paragraph{Build(f)} At Step~\ref{callBuild}, 
the algorithm calls a subroutine to build a certificate
of minimality.  Recall the certificate is either a homomorphism
$f: F \to G$ or a pair of homomorphisms $f_1: F_1 \to G$ and $f_2: F_2 \to G$ 
to certify that $G \not\to D'$ for any $D' < D$ where $g: G \to D$ 
is returned by the algorithm. (Here $D' < D$ is the homomorphism order on 
$\mathcal{D}$.  Thus, $D' \to D$ but $D \not\to D'$.)

Consider the post-processing phase of the algorithm at Step~\ref{postProc}.
If $B_i \cup R_i = \emptyset$, then let $u \in I_i$.  By Proposition~\ref{prop:Invariant},
$u$ is the centre vertex of an alternating path of length $2i-2$, i.e. 
there exists $f: F_{2i-2} \to G$ which is the certificate of minimality for
$g: G \to D_{2i-1}$.  Otherwise, consider the next case: Elseif $R_i = \emptyset$.
Let $u \in B_i$.  Then $u$ is incident with blue edges in the final iteration of
the loop.  In particular, there is a vertex $v \in B_i$ such that $uv$ is a blue edge.
Both $u$ and $v$ are termini of alternating paths of length $i-1$ whose final edges
are red.  These two paths together with the edge $uv$ allow us to define
$f: F^B_{2i-1} \to G$.  This is a certificate of minimality for $g: G \to D^{B}_{2i-1}$.
Similarly, in the next case (Elseif $B_i = \emptyset$) we can find a 
certificate of minimality $f: F^R_{2i-1} \to G$ for $g: G \to D^R_{2i-1}$.
Finally, in the last case $R_i \neq \emptyset$ and $B_i \neq \emptyset$.  Using the
arguments above, we can find $f_B: F^B_{2i-1} \to G$ and $f_R: F^R_{2i-1} \to G$.
This pair of maps is a certificate of minimality for $g: G \to D_{2i}$.

\section{Proofs of Theorems~\ref{thm:main1} and~\ref{thm:main2}}

\emph{Proof of Theorem~\ref{thm:main1}:}
Suppose $F_{2k} \to G$ and suppose to the contrary $G \to D_{2k}$.
Then by composition, $F_{2k} \to D_{2k}$ contrary to 
Lemma~\ref{lem:FnotMap}. Similarly, if $F^R_{2k-1} \to G$,
then $G \not\to D^B_{2k-1}$, and if $F^B_{2k-1} \to G$, then
$G \not\to D^{R}_{2k-1}$.

Conversely, suppose that $F_{2k} \not\to G$.  Run the algorithm in
Figure~\ref{fig:dualityalgorithm}.  First observe that $F_{2k}$ maps
to any closed alternating walk.  We can conclude that $G$ does not
contain a closed alternating walk. In particular, the algorithm
does not terminate at Step~\ref{step:Smooth}, but rather the main 
loop terminates with a map $g: G \to D_{2i}$.  In the post processing
step we construct at least one of the following maps: $f: F_{2i-2} \to G$,
$f_R: F^R_{2i-1} \to G$, or $f_B: F^B_{2i-1} \to G$.
By Proposition~\ref{prop:subgraphs} and the assumption $F_{2k} \not\to G$,
we have $2i-2 < 2k$.  Hence, $2i \leq 2k$ and $G \to D_{2i} \to D_{2k}$.

Suppose $F^R_{2k-1} \not\to G$.  Again the algorithm cannot discover a
smooth subgraph in $G$, so $g: G \to D_{2i}$ for some $i$.  If $i < k$,
then $G \to D_{2i} \to D^B_{2k-1}$ as required.  The map(s) constructed
in the post processing step ensures $F_{2i-2} \to G$ and thus
$i \leq k$, so $i=k$.  This
implies the certificate of minimality from Step~\ref{postProc} is
either $f: F_{2k-2} \to G$ (in which case $g: G \to D_{2k-1}$ is
returned by the algorithm) or $f: F^B_{2k-1} \to G$ (in which
case $g: G\to D^B_{2k-1}$ is returned by the algorithm).  In
both cases $G \to G^B_{2k-1}$ as required.  The case
$F^B_{2k-1} \not\to G$ is analogous. \qedsymbol

\medskip

\noindent \emph{Proof of Theorem~\ref{thm:main2}:}
Note for any even positive integer $2k$, $\mathcal{F}_{2k} = \{F_{2k}\}$. 
Thus $G \to D_{2k}$ if and only if for all 
$F \in \mathcal{F}_{2k}, F \not\to G$ by Theorem~\ref{thm:main1}. 
For the odd integer $2k-1$, $\mathcal{F}_{2k-1} = \{F^R_{2k-1},F^B_{2k-1}\}$. 
Suppose $F^R_{2k-1} \not\to G$ and $F^B_{2k-1} \not\to G$.
At the termination of the algorithm $g : G \to D_{2i}$.  If $2i < 2k-1$,
then $G \to D_{2k-1}$ as required.  Otherwise, it must be the
case $2i = 2k$.  Consequently, the algorithm must return $F_{2k-2} \to G$ 
as the certificate of minimality (from Case 1 of the post processing).
Hence $G \to D_{2k-1}$ as required. \qedsymbol

\begin{defn}
Let G and H be edge-coloured graphs. The \emph{categorical product} $G \times H$ is the graph 
with vertex set $V(G \times H) = V(G) \times V(H)$, and edges $((u, u'),(v,v'))$ of colour $i$ 
if and only if $u'$ is adjacent with $v'$ with colour $i$ and $u$ is adjacent with $v$ with colour $i$.
\end{defn}

The following corollary follows from~\cite{finiteduality} and Theorem~\ref{thm:main2}.
\begin{theorem}[Ne\v{s}et\v{r}il and Tardif~\cite{finiteduality}]\label{thm:finite}
The pair $( \mathcal{F}, D)$ is a finite homomorphism duality if and only if
$$
D \sim \prod_{F \in \mathcal{F} } D(F).
$$
\end{theorem}

It is well known that $G \to X \times Y$ if and only if $G \to X$ and $G \to Y$.
Hence, the following corollary shows when $G \to D^R_{2k-1}$ and $G \to D^B_{2k-1}$ 
it must be the case that $D \to D_{2k-1}$.  (This follows from our certificates
of minimality.)  We conclude that if $G \to D \in \mathcal{D}$ there is a
well defined minimum element $D'$ such that $G \to D'$.  

\begin{cor}\label{cor:homequivD}
The \ec $D_{2k-1}$ can be expressed as 
$$D_{2k-1} \sim D^R_{2k-1} \times D^B_{2k-1}.$$
\end{cor}

\section{Future Work}

As stated in the introduction, our aim in this paper was to find the
duals of alternating paths and examine them for similarities with transitive
tournaments.  On the one hand we have that any \ec graph  $G$ without a closed
alternating walk admits a homomorphism to $\mathcal{D}$.  The ordering
on $\mathcal{D}$ ensures that there is a well defined minimum element
in $\mathcal{D}$ to which $G$ maps.  This is analogous to acyclic digraphs
admitting a homomorphism to a minimum transitive tournament.  That is,
$\mathcal{D}$ maybe viewed as a \emph{calibrating family}
as defined by Hell and Ne\v{s}et\v{r}il~\cite{thebook}.

On the other hand, our attempts to generalize other results remain open.
For example, given an undirected graph, 
the Gallai-Hasse-Roy-Vitaver Theorem (\cite{Gallai66,Hasse64,Roy67,Vitaver62}
see also Corollary 1.21~\cite{thebook}) gives
the well known connection between colourings of the graph and
acyclic orientations of the graph.  Specifically using the duality
pair $(\vec{P}_{n+1},T_n)$ the following is immediate.
\begin{theorem}\label{thm:ghrv}
A graph $G$ is $n$-colourable if and only if there is an orientation $D$
of $G$ such that $\vec{P}_{n+1} \not\to D$.
\end{theorem}

The connection between orientations and colourings of a graph also
appears in Minty's Theorem.  Let $D$ be an orientation of a graph $G$. 
Given a cycle $C$ of $D$, we select a direction of traversal.
Let  $|C^+|$ denote the set of forward arcs 
of $C$, and $|C^-|$ denote the set of backward arcs of $C$ with respect
to the traversal direction. 
Define $\tau(C) = \max\{|C|/|C^-|, |C|/|C^+|\}$ 
to be the \emph{imbalance} of $C$. 

\begin{theorem}[Minty~\cite{Minty62}]\label{thm:minty}
For any graph $G$, 
$$
\chi (G) = \min_{D} \{ \lceil \max \{\tau(C): \emph{C is a cycle of D}\} \rceil 
%: \emph{D is an acyclic orientation of G} 
\}
$$
where the minimum is over all acyclic orientations $D$ of $G$.
\end{theorem}

We wonder if these theorems can be generalized. 

\begin{quest}
Are there analogues to  Theorems~\ref{thm:ghrv} and~\ref{thm:minty} 
using the family $\mathcal{D}$ for \ec graphs?
\end{quest}

On the surface the idea of assigning colours to
the edges of a graph and minimizing the length of an alternating path 
will not work. Assigning all edges to blue produces a \ec with 
longest alternating path of length one and a homomorphism to $D_1^B$,
independently of the chromatic number of the graph.

One may also consider the duals of other \ec paths.  The variety
(closure under products and retracts) of these duals gives the
set of \ec graphs admitting a \emph{near unanimity function} of
order $3$: $NU_3$.

\begin{quest}
Can we easily describe duals of other \ec paths?
\end{quest}

Finally, we can also study the case with more edge colours or even $(m,n)$-mixed
graphs (graphs with $m$ edge sets and $n$ arc sets~\cite{NesetrilRaspaud00}).
The introduction of arcs and edges changes the category under question with
significant implications on duality pairs. Replacing unordered coloured edges
with a pair of oppositely directed arcs (of the same colour) to move our
setting to \ec digraphs causes our main results to no longer hold.  
For example, $(F_1^R,D_1^B)$ is no longer a duality pair.  A single red
arc is a counterexample.  In fact, in this setting $F_1^R$ has no dual.

\begin{quest}
Can the duality results here be generalized to $(m,n)$-mixed graphs?
\end{quest}

\end{document}